\documentclass[12pt]{amsart}
\usepackage{amssymb,amscd,amsthm,amsmath,color, xcolor}
\usepackage{fullpage}
\usepackage{tikz-cd}

\newcommand{\PP}{\mathbb{P}}
\newcommand{\pmn}{\PP^m\times\PP^n}
\newcommand{\OO}{\mathcal{O}}
\newcommand{\II}{\mathcal{I}}

\newcommand{\FF}{\mathbb{F}}

\newcommand{\DD}{\mathcal{D}}
\newcommand{\GG}{\mathbb{G}}

\newcommand{\ZZ}{\mathbb{Z}}

\newcommand{\eE}{\mathcal{E}}

\newcommand{\ev}{\operatorname{ev}}

\newcommand{\codim}{\operatorname{codim}}

\newcommand{\xX}{\mathcal{X}}
\newcommand{\yY}{\mathcal{Y}}

\newcommand{\Pic}{\operatorname{Pic}}

\theoremstyle{plain}
\newtheorem{lemma}{Lemma}[section]
\newtheorem*{theorem*}{Theorem}
\newtheorem*{lemma*}{Lemma}
\newtheorem*{proposition*}{Proposition}
\newtheorem*{conjecture*}{Conjecture}
\newtheorem*{corollary*}{Corollary}
\newtheorem*{problem*}{Problem}
\newtheorem{theorem}[lemma]{Theorem}

\newtheorem{proposition}[lemma]{Proposition}

\theoremstyle{definition}
\newtheorem{definition}[lemma]{Definition}

\newtheorem{remark}[lemma]{Remark}

\begin{document}

\title{Algebraic hyperbolicity of very general hypersurfaces in products of projective spaces}

\author[Wern Yeong]{Wern Yeong}
\email{wyeong@nd.edu}
\address{Department of Mathematics \\ University of Notre Dame, Notre Dame, IN 46556}

\begin{abstract} 
We study the algebraic hyperbolicity of very general hypersurfaces in $\pmn$ by using three techniques that build on past work by Ein, Voisin, Pacienza, Coskun and Riedl, and others. As a result, we completely answer the question of whether or not a very general hypersurface of bidegree $(a,b)$ in $\pmn$ is algebraically hyperbolic, except in $\PP^3\times \PP^1$ for the bidegrees $(a,b)= (7,3), (6,3)$ and $(5,b)$ with $b\geq 3.$ As another application of these techniques, we improve the known result that very general hypersurfaces in $\PP^n$ of degree at least $2n-2$ are algebraically hyperbolic when $n\geq 6$ to $n \geq 5$, leaving $n=4$ as the only open case.
\end{abstract}

\maketitle

\section{Introduction}\label{sec-introduction}

A complex projective variety is \emph{algebraically hyperbolic} if for some $\varepsilon>0$ and some ample divisor $H$, every integral curve $C$ that lies in the variety satisfies
\begin{equation}
\label{eqn-definition}
    2g(C)-2 \geq \varepsilon\cdot \deg_H C,
\end{equation}
where $g(C)$ is the geometric genus of the curve. It follows that varieties that contain any rational or elliptic curves are not algebraically hyperbolic. 

Algebraic hyperbolicity was introduced by Demailly in \cite{Demailly} as an algebraic analogue to Kobayashi hyperbolicity for complex manifolds, which was shown to be equivalent to Brody hyperbolicity when the manifold is compact \cite{Brody}. A compact complex manifold $X$ is \emph{Brody hyperbolic} if all holomorphic maps $\mathbb{C}\rightarrow X$ are constant, i.e. $X$ contains no entire curves. Demailly \cite{Demailly} proved that for smooth projective varieties Kobayashi hyperbolicity implies algebraic hyperbolicity, and conjectured that the converse holds.

The algebraic and Brody hyperbolicity of very general degree $d$ hypersurfaces $X_d\subseteq\PP^n$ is well-studied. (See for instance the surveys \cite{CoskunSurvey, DemaillySurvey, VoisinSurvey}.) For $n=3$, Xu \cite{Xu} proved that very general surfaces $X_d\subseteq \PP^3$ of degree $d\geq 6$ are algebraically hyperbolic by using a degeneration argument. Then Coskun and Riedl \cite{CoskunRiedl} improved Xu's bound to $d\geq 5$ by considering the degrees of surface scrolls that contain any particular curve in $X_d$. Their result completes the classification in $\PP^3$ since surfaces of degree at most $4$ contain rational curves. For larger $n$, Clemens \cite{Clemens86} and Ein \cite{Ein} proved that $X_d$ is algebraically hyperbolic when $d\geq 2n$ for $n\geq 4$. The bound was subsequently improved by Voisin \cite{Voisin, Voisincorrection} to $d\geq 2n-1$ for $n\geq4$, and then by Pacienza \cite{Pacienza2} and Clemens and Ran \cite{ClemensRan} to $d\geq 2n-2$ for $n\geq 6$. We confirm in this article that very general degree 8 hypersurfaces in $\PP^5$ are algebraically hyperbolic, hence improving the previously-known result to $d\geq 2n-2$ for $n\geq 5$ (see Theorem \ref{thm-pn}). When $d\leq 2n-3$, $X_d$ contains lines and hence, is not algebraically hyperbolic. Thus, the only remaining open question in $\PP^n$ is the case of sextic threefolds. 

The hyperbolicity of very general hypersurfaces has also been studied in other ambient varieties besides $\PP^n$. Haase and Ilten \cite{HaaseIlten} studied very general surfaces in Gorenstein toric threefolds by building on the focal loci techniques of Chiantini and Lopez \cite{CL}. As an application, they almost completely classified very general surfaces in the threefolds $\PP^2\times \PP^1$, $\PP^1 \times \PP^1 \times \PP^1$, $\FF_e\times \PP^1$ and the blowup of $\PP^3$ at a point. Then, Coskun and Riedl completed this classification in \cite{CoskunRiedlsurfaces} as a result of their techniques for three-dimensional complex projective varieties that admit a group action with dense orbit. More recently, Robins \cite{Robins} almost completely classified very general surfaces in smooth projective toric threefolds with Picard rank 2 or 3 by combining Haase and Ilten's focal loci techniques with a study of the combinatorics of such threefolds.

In this paper, we study very general hypersurfaces in $\pmn$ and prove the following main theorem. We assume throughout that $m\geq n.$

\begin{theorem}\label{thm-pmxpn}
	A very general hypersurface in $\pmn$ of bidegree $(a,b)$ is 
	\begin{enumerate}
	    \item algebraically hyperbolic if
	    \begin{enumerate}
	        \item $m+n\geq5$: $a\geq 2m+n-2$ and $b\geq m+2n-2$; or
	        \item $(m,n)=(2,2)$: $a\geq 5$ and $b\geq 5$; or
	        \item $(m,n)=(3,1)$: $a\geq 6$ and $b\geq 4$, or $a\geq 8$ and $b= 3$,
	    \end{enumerate}
	    \item not algebraically hyperbolic if
	    \begin{enumerate}
	        \item $m+n\geq5$: either $a< 2m+n-2$ or $b< m+2n-2$; or
	        \item $(m,n)=(2,2)$: either $a< 5$ or $b< 5$; or
	        \item $(m,n)=(3,1)$: either $a< 5$ or $b< 3$.
	    \end{enumerate}
	\end{enumerate}
	This completely resolves algebraic hyperbolicity of very general hypersurfaces in $\pmn$, except in $\PP^3\times \PP^1$ for bidegrees $(a,b)= (7,3), (6,3)$ and $(5,b)$ with $b\geq 3.$
\end{theorem}
 
Note that the classification of very general surfaces in $\PP^2\times \PP^1$ was already done by Haase and Ilten \cite{HaaseIlten} and Coskun and Riedl \cite{CoskunRiedlsurfaces}, while hypersurfaces in $\PP^1\times \PP^1$ are curves, which are easy to classify. We give a description of the remaining open cases in $\PP^3\times \PP^1$ in Remark \ref{rmk-open}.

\subsection{Organization} In \S\ref{sec-setup}, we introduce the setup developed by Ein, Voisin, Pacienza, and others. In \S\ref{sec-tech}, we present the three techniques that are used in \S\ref{sec-main} to prove the main theorem. As another application of these techniques, we give a streamlined proof of the algebraic hyperbolicity of very general hypersurfaces in $\PP^n$ of degree at least $2n-1$ when $n\geq4$ and degree at least $2n-2$ when $n\geq 5$.

\subsection{Acknowledgements} I would like to thank my adviser Eric Riedl for suggesting this problem to me, and for many valuable discussions and revisions of my drafts.

\section{Setup}\label{sec-setup}

We use a similar setup as \cite{ClemensRan, CoskunRiedl, CoskunRiedlsurfaces, Pacienza, Pacienza2, Voisin}. Denote by $\pi_1$ and $\pi_2$ the projection maps from $\pmn$. Let $H_1=\pi_1^*\OO_{\PP^m}(1)$ and $H_2=\pi_2^*\OO_{\PP^n}(1)$. Then $\Pic(\pmn)=\ZZ\cdot H_1\oplus \ZZ\cdot H_2$. For $a,b>0$, denote the line bundle $aH_1+bH_2$ by $\eE$. Suppose that a very general section $X$ of $\eE$ contains a curve $Y$ of geometric genus $g$ and degree $e$ with respect to the ample class $H_1+H_2$. Denote $V=H^0(\pmn, \eE)$, and let $\xX_V\subseteq (\pmn)\times V$ be the universal hypersurface of bidegree $(a,b)$. Let $T\subseteq \text{Hilb}(\pmn)\times V$ be an irreducible variety such that 
\begin{enumerate}
	\item a general element $(Y, F)$ of $T$ projects to a point in $\text{Hilb}(\pmn)$ representing a curve $Y$ of geometric genus $g$ and $(H_1+H_2)$-degree $e$ that is contained in the hypersurface defined by $F$, and 
	\item the projection map $T\rightarrow V$ is dominant, and
	\item $T$ is stable under the $G:=GL_{m+1}\times GL_{n+1}$-action on $\text{Hilb}(\pmn)\times V.$ 
\end{enumerate}

Let $\yY_T\subseteq \xX_T := \xX_V \times_V T$ denote the universal curve over $T$. By (3), there is a $G$-invariant subvariety $U\subseteq T$ such that $U\rightarrow V$ is \'etale, and we restrict the universal curve to $\yY_U$. Finally after taking a $G$-equivariant resolution of $\yY_U$ and restricting $U$ to a $G$-invariant open set, we obtain a smooth family $\yY\rightarrow U$ whose fibers are smooth curves mapping to curves of geometric genus $g$ and $(H_1+H_2)$-degree $e$. Similarly, pull back $\xX_T$ to $U$ to get a smooth family $\xX\rightarrow U$ of hypersurfaces. Denote the projection maps by $p_1:\xX\rightarrow U$ and $p_2:\xX\rightarrow \pmn$. Denote by $h:\yY\rightarrow \xX$ the natural map, which is generically injective. Note that 
\begin{equation}\label{eqn-codimyx}
    \codim (\yY\subseteq \xX) = m+n-2.
\end{equation}

By $G$-invariance, we can define the \emph{vertical tangent sheaves} $T_{\xX/\pmn}$ and $T_{\yY/\pmn}$ by the following short exact sequences:
$$0\rightarrow T_{\xX/\pmn}\rightarrow T_\xX \rightarrow p_2^*T_{\pmn}\rightarrow 0$$
$$0\rightarrow T_{\yY/\pmn}\rightarrow T_\yY \rightarrow h^*p_2^*T_{\pmn}\rightarrow 0$$

Let $t\in U$ be a general element. Denote the fibers of $h^*\yY$ and $\xX$ over $t$ by $Y_t$ and $X_t$ respectively, and the restriction of $h$ by $h_t:Y_t\rightarrow X_t.$ Let us also define the normal bundles $N_{h/\xX}$ and $N_{h_t/X_t}$ by the following short exact sequences:
$$0\rightarrow T_\yY \rightarrow h^*T_\xX \rightarrow N_{h/\xX} \rightarrow 0$$
$$0\rightarrow T_{Y_t}\rightarrow h_t^*T_{X_t} \rightarrow N_{h_t/X_t}\rightarrow 0$$

Taking degrees in the second short exact sequence above, we have
\begin{equation}\label{eqn-degree}
    2g(Y_t)-2-K_{X_t}\cdot h_t(Y_t)=\deg N_{h_t/X_t}.  
\end{equation}
Hence, one way to confirm algebraic hyperbolicity is to find a suitable lower bound for $\deg N_{h_t/X_t}$.

We denote by $M_\eE$ the \emph{Lazarsfeld-Mukai bundle associated to $\eE$}, which is defined by the short exact sequence
$$0 \rightarrow M_\eE \rightarrow H^0(\pmn,\eE) \otimes \OO_{\pmn} \xrightarrow{\ev} \eE \rightarrow 0.$$
We can define similarly the Lazarsfeld-Mukai bundles $M_{H_1}$ and $M_{H_2}$ associated to $H_1$ and $H_2$ respectively. These bundles play a central role in the proof techniques in \S\ref{sec-tech}.

The following result follows immediately from Proposition 2.1 in \cite{CoskunRiedlsurfaces} and summarizes the relationship between the vertical tangent sheaves, the normal bundles, and the Lazarsfeld-Mukai bundle $M_\eE$.

\begin{proposition}[cf Proposition 2.1 in \cite{CoskunRiedlsurfaces}]\label{prop-main}\;
    \begin{enumerate}
        \item $N_{h/\xX}\vert_{Y_t}\simeq N_{h_t/X_t}$.
        \item The quotient $\left(h^*T_{\xX/\pmn}\right)\big/\left(T_{\yY/\pmn}\right)$ of vertical tangent sheaves is isomorphic to $N_{h/\xX}$.
        \item The vertical tangent sheaf $T_{\xX/\pmn}$ on $\xX$ is isomorphic to the pullback $p_2^*M_\eE$ of the Lazarsfeld-Mukai bundle.
    \end{enumerate}
\end{proposition}

\section{Proof techniques}\label{sec-tech}

In order to estimate the positivity of $N_{h_t/X_t},$ we start with the notion of a section-dominating collection of line bundles introduced in \cite{CoskunRiedlsurfaces}.

\begin{definition}[cf Definition 2.3 in \cite{CoskunRiedlsurfaces}]\label{def-secdom}
    Let $\eE$ be a vector bundle on a smooth, complex projective variety $A$. A collection of non-trivial, globally generated line bundles $L_1,\ldots,L_u$ is called a \emph{section-dominating collection} of line bundles for $\eE$ if 
    \begin{enumerate}
        \item $\eE\otimes L_i^\vee$ is globally generated for every $1\leq i \leq u$, and
        \item the map 
        $$\bigoplus^u_{i=1} \left( H^0(L_i\otimes \II_p) \otimes H^0(\eE\otimes L_i^\vee) \right)\rightarrow H^0(\eE\otimes \II_p) $$ 
        is surjective at every point $p\in A$.
    \end{enumerate}
\end{definition}

We can view a section-dominating collection of line bundles for $\eE$ as a collection of simpler building blocks for $\eE$, in the sense of the following proposition.

\begin{proposition}[cf Proposition 2.6 in \cite{CoskunRiedlsurfaces}]\label{prop-secdom}
    Let $\eE$ be a globally generated vector bundle and $M_\eE$ the Lazarsfeld-Mukai bundle associated to $\eE.$ Let $L_1,\ldots,L_u$ be a section-dominating collection of line bundles for $\eE.$ Then for some integers $s_i$, there is a surjection 
    \begin{equation}\label{eqn-3.2}
        \oplus_{i=1}^u M_{L_i}^{\oplus s_i} \longrightarrow M_\eE
    \end{equation}
    induced by multiplication by some choice of basis elements of $H^0(\eE\otimes L_i^\vee)$ for $0\leq i \leq u.$
\end{proposition}

Example 2.4 in \cite{CoskunRiedlsurfaces} shows that on $A=\pmn$, the line bundles $H_1,H_2$ form a section-dominating collection for $\eE=aH_1+bH_2.$ By Proposition \ref{prop-secdom}, this gives a surjection (\ref{eqn-3.2}) induced by multiplication by generic polynomials in $H^0(\eE\otimes H_i^\vee)$ for $i=1,2$. We can use this surjection to obtain a first pair of lower bounds $a\geq a_0$ and $b\geq b_0$ such that a very general hypersurface of bidegree $(a,b)$ is algebraically hyperbolic. 

\begin{lemma}\label{lem-bounds}
    Let $h_t:Y_t\rightarrow X_t$ be a curve in a very general hypersurface $X_t\subseteq \pmn$ of bidegree $(a,b)$. Let $N$ be a vector bundle on the curve. Suppose that there is a surjective map
	\begin{equation*}
		\beta:M_{H_1}\;^{\oplus t_1}\oplus M_{H_2}\;^{\oplus t_2}\longrightarrow N
	\end{equation*} for some $t_1,t_2$. Then, 
	\begin{equation}\label{eqn-bounds1}
        \deg N\geq -t_1\cdot ({Y_t}\cdot H_1)-t_2\cdot ({Y_t}\cdot H_2).
    \end{equation}
	In particular, setting $N=N_{h_t/X_t},$ we obtain
    \begin{equation}\label{eqn-bounds2}
        2g(Y_t)-2\geq (a-(m+1+t_1))\cdot ({Y_t}\cdot H_1)+(b-(n+1+t_2))\cdot ({Y_t}\cdot H_2).    
    \end{equation}
\end{lemma}

\begin{proof}
    Since $\beta$ is surjective, we have
    \begin{equation*}
        \begin{split}
            \deg N&\geq \deg \beta\left(M_{H_1}\;^{\oplus t_1})+ \deg \beta(M_{H_2}\;^{\oplus t_2}\right)\\
            &\geq t_1\cdot (-\deg H_1\vert_{Y_t}) + t_2\cdot (-\deg H_2\vert_{Y_t})
        \end{split}
    \end{equation*}
	where the second inequality follows from Proposition 2.7 in \cite{CoskunRiedlsurfaces}. If $N=N_{h_t/X_t}$, then by (\ref{eqn-degree}), we obtain
	\begin{equation*}
	    \begin{split}
	        &2g(Y_t)-2\\ 
	        &= (K_{X_t}\cdot Y_t) + \deg N_{h_t/X_t}\\
	        &\geq (a-(m+1+t_1))\cdot (Y_t\cdot H_1) + (b-(n+1+t_2))\cdot (Y_t\cdot H_2)
        \end{split}
	\end{equation*}
\end{proof}

\begin{proposition}\label{prop-0}
    Suppose $a\geq 2m+n$ and $b\geq m+2n$, then a very general hypersurface $X\subseteq \PP^m\times \PP^n$ of bidegree $(a,b)$ is algebraically hyperbolic.
\end{proposition}
	
\begin{proof}
	Suppose $X$ is parametrized by $t\in U$ and let $h_t:Y_t\rightarrow X_t$ be a curve in $X$. By Propositions \ref{prop-main} and \ref{prop-secdom}, we have a surjection
	\begin{equation}\label{eqn-secdom}
		\alpha:M_{H_1}\;^{\oplus s_1}\oplus M_{H_2}\;^{\oplus s_2} \longrightarrow M_\eE \longrightarrow N_{h_t/X_t}
	\end{equation} 
	for some integers $s_1,s_2.$ Since rank $N_{h_t/X_t}=m+n-2,$ we can take $s_1=s_2=m+n-2$.
	
	Therefore, by (\ref{eqn-bounds2}), we have
	\begin{equation*}
	    \begin{split}
	        &2g(Y_t)-2\\
	        &\geq ((a-m-1)-(m+n-2))\cdot (Y_t\cdot H_1) + ((b-n-1)-(m+n-2))\cdot (Y_t\cdot H_2)\\
	        &=(a-(2m+n-1))\cdot (Y_t\cdot H_1) + (b-(m+2n-1))\cdot (Y_t\cdot H_2)
        \end{split}
	\end{equation*}
	which is at least $(Y_t\cdot H_1)+(Y_t\cdot H_2)$ when $a\geq 2m+n$ and $b\geq m+2n$.
\end{proof}

\begin{remark}\label{rmk-ein}
    The proof of Proposition \ref{prop-0} is essentially the approach taken by Ein in \cite{Ein, Ein2} to prove that if $X\subseteq\PP^n$ is a generic complete intersection of type $(d_1,\ldots,d_k)$ satisfying $\sum_i d_i\geq 2n-k+1$, then the desingularization of every subvariety of $X$ is of general type. In particular, such varieties $X$ are algebraically hyperbolic. The above proof is also in the same spirit as Theorem 3.6 in \cite{HaaseIlten}, which was proved using focal loci techniques.
\end{remark}

The lower bounds in Proposition \ref{prop-0} are within 1 or 2 of the lower bounds stated in Theorem \ref{thm-pmxpn}(a). In the rest of this section, we present two techniques that improve the bounds in Proposition \ref{prop-0}.

In the proof of Proposition \ref{prop-0}, one sees that it is possible for the surjection (\ref{eqn-secdom}) to be achieved with smaller $s_1,s_2$, which would give us better control of the positivity of $N_{h_t/X_t}$. This motivates the following definition.

\begin{definition}\label{def-s1s2}
    We say that a curve $h_t:Y_t\rightarrow X_t$ is \emph{of type $(s_1,s_2)$} if the induced map 
    \begin{equation*}
        \alpha: M_{H_1}^{\oplus s_1} \oplus M_{H_2}^{\oplus s_2} \longrightarrow N_{h_t/X_t}
    \end{equation*}
    is surjective, with no summand having torsion image. In particular, $s_1+s_2\leq m+n-2$.
\end{definition}

Note that a curve may be of several different types, since the integers $s_1,s_2$ as defined above may not be unique. 

\begin{remark}\label{rem-ClemensRan} (cf \cite{Clemens03, ClemensRan})
    Due to the genericity of the polynomials inducing the map $\alpha$, for a fixed $0\leq u_2\leq s_2$, the rank $\gamma(i)$ of $\alpha\left(M_{H_1}^{\oplus i} \oplus M_{H_2}^{\oplus u_2}\right)$ is a strictly increasing function for $0\leq i \leq s_1$, while the increase $\gamma(i+1)-\gamma(i)$ in rank is a non-decreasing function of $i$  for $0\leq i \leq s_1-1.$ The analogous statement with a fixed $0\leq u_1\leq s_1$ and varying $0\leq j \leq s_2$ is true as well.
\end{remark}

Now, we give an outline of the next two techniques. When $s_1+s_2$ is large, i.e. $s_1+s_2> (m+n-2)/2$, there is a summand $M_{H_i}$ whose image in a quotient of $N_{h_t/X_t}$ induced by $\alpha$ has rank one. In this case, \S\ref{sec-scroll} shows that the curve $h_t:Y_t\rightarrow X_t$ lies in a special surface in $\pmn$ made of ``lines'' (see \S\ref{sec-scroll} for definition) passing through each point of the curve. This surface is described in Lemma \ref{lem-scroll}. This method produces an improvement in the estimate for the positivity of $N_{h_t/X_t}$, which then improves the lower bounds for $a$ and $b$ in Proposition \ref{prop-0} by 1. 

When $s_1$ (or $s_2$) achieves the maximum, i.e. $s_1$ (or $s_2$) equals $m+n-2$, $N_{h_t/X_t}$ modulo torsion is a direct sum of $m+n-2$ rank-one images of $M_{H_1}$ (or $M_{H_2}$) via $\alpha$. In this case, \S\ref{sec-osc} shows that the curve lies in a special locus on the hypersurface, which turns out to be a curve of general type when $a=2m+n-2$ (or $b=m+2n-2$). This method improves the lower bounds for $a$ and $b$ in Proposition \ref{prop-0} by 2 when $m+n\geq 5$. When $(m,n)=(3,1)$, this method improves the lower bound for $b$ in Proposition \ref{prop-0} by 2. 

\subsection{Scroll considerations}\label{sec-scroll}

Let us define a \emph{$\PP^m$-line} (resp. \emph{$\PP^n$-line}) in $\pmn$ to mean an integral curve $L$ whose numerical class is $H_1^{m-1}H_2^n$ (resp. $H_1^{m}H_2^{n-1}$), and a \emph{line} may refer to either a $\PP^m$-line or a $\PP^n$-line. A surface in $\pmn$ is called a \emph{$\PP^m$-scroll} (resp. \emph{$\PP^n$-scroll}) if there is a $\PP^m$-line (resp. $\PP^n$-line) through every point of the surface.

\begin{lemma}[cf Lemma 2.12 in \cite{CoskunRiedlsurfaces}]\label{lem-scroll}
	A rank-one quotient $Q$ of $M_{H_1}\vert_{Y_t}$ induces a \linebreak $\PP^m$-scroll containing $Y_t$ of $\pi_1^*\OO_{\PP^m}(1)$-degree equal to $\deg Q + (Y_t\cdot H_1).$
\end{lemma}

\begin{proof}
	Recall the Euler sequence on $\PP^m$:
	$$0 \rightarrow \Omega_{\PP^m}(1)\rightarrow \OO^{m+1}_{\PP^m}\rightarrow \OO_{\PP^m}(1)\rightarrow 0.$$
	The pullback of $\Omega_{\PP^m}(1)$ via $\pi_1$ is isomorphic to the Lazarsfeld-Mukai bundle $M_{H_1}$. By pulling back the above short exact sequence via $\pi_1$ and then restricting to the curve $Y_t\subseteq X_t$, we get 
	$$0 \rightarrow M_{H_1}\vert_{Y_t} \rightarrow H^0(\pmn,\pi_1^*\OO_{\PP^m}(1))\otimes \OO_{Y_t} \rightarrow \pi_1^*\OO_{\PP^m}(1)\vert_{Y_t} \rightarrow 0.$$
	We define $S$ and $Q'$ to be the sheaves that make this commutative diagram exact:
	\begin{center}
		\begin{tikzcd}
		&0 \arrow[dr] &0 \arrow[d]\\
		& &S \arrow[d] \arrow[dr] & &0\\
		&0 \arrow[r] &M_{H_1}\vert_{Y_t} \arrow[r] \arrow[d] &\OO_{Y_t}^{m+1} \arrow[r] \arrow[dr, "(*)"] &\pi_1^*\OO_{\PP^m}(1)\vert_{Y_t} \arrow[r] \arrow[u] &0\\
		& &Q \arrow[d] & &Q'\arrow[u] \arrow[dr]\\
		& &0 & & &0
		\end{tikzcd}
	\end{center}
	The triangle ($*$) induces a map $\gamma: Y_t\rightarrow \GG(1,m)$ that sends a point $p\in Y_t$ to a line in $\PP^m$ given by the rank-two vector bundle $Q'$, which contains $\pi_1(p)$. $\gamma$ lifts to a map that sends $p\in Y_t$ to the $\PP^m$-line $\pi_1^{-1}(\gamma(p))\cap \pi_2^{-1}(\pi_2(p))$ which contains $p$, giving a $\PP^m$-scroll containing $Y_t$. The $\pi_1^*\OO_{\PP^m}(1)$-degree of this surface follows from taking degrees in the above diagram.
\end{proof}

The following is the analogous result with $H_1$ and $H_2$ switched.

\begin{lemma}\label{lem-scroll2}
	A rank-one quotient $Q$ of $M_{H_2}\vert_{Y_t}$ induces a $\PP^n$-scroll containing $Y_t$ of \linebreak $\pi_2^*\OO_{\PP^n}(1)$-degree equal to $\deg Q + (Y_t\cdot H_2).$
\end{lemma}

To give an idea of their application to the proof of Theorem \ref{thm-pmxpn}, suppose that $h_t:Y_t\rightarrow X_t$ is a curve of type $(s_1,s_2)=(m+n-2,0)$. Then the map
$$M_{H_1}\rightarrow N_{h_t/X_t}\Big/\alpha(M_{H_1}^{\oplus m+n-3})$$
induced by (\ref{eqn-secdom}) has a rank-one image $Q$, which by Lemma \ref{lem-scroll} produces a $\PP^m$-scroll containing the curve whose intersection number with the class $H_1^2$ equals $\deg Q + (Y_t\cdot H_1)$. By comparing the numerical classes of the curve, the hypersurface and the $\PP^m$-scroll, we obtain a lower bound for $\deg Q$, which in turn gives the following lower bound for $\deg N_{h_t/X_t}$ by (\ref{eqn-bounds1}):
\begin{equation*}
    \deg N_{h_t/X_t}\geq \deg Q + \deg \alpha(M_{H_1}^{\oplus m+n-3})\geq \deg Q -(m+n-3)\cdot (Y_t\cdot H_1).
\end{equation*}

\begin{remark}\label{rmk-lines}
    The above scroll considerations were first introduced in \cite{CoskunRiedl} to show that a curve $C$ in a very general degree $d$ surface $X\subseteq\PP^3$ satisfies the inequality 
    $$ 2g(C)-2\geq \left(d-5+\frac{1}{d}\right)\cdot (C\cdot H),$$ hence $X$ is algebraically hyperbolic if $d\geq 5$. Similarly, when applied to curves $C$ contained in very general degree $d$ hypersurfaces $X\subseteq\PP^n$ of higher dimensions, the scroll method yields the inequality
    $$ 2g(C)-2\geq \left(d-(2n-1)+\frac{1}{d}\right)\cdot (C\cdot H).$$ Hence, it gives an alternate proof (see Theorem \ref{thm-pn}) that $X$ is algebraically hyperbolic if $d\geq 2n-1$ and $n\geq 4$, which is a result first obtained in \cite{Voisin, Voisincorrection}. Prior to \cite{CoskunRiedl}, Clemens-Ran used a similar scroll construction in \S4 of \cite{ClemensRan} to discount the existence of curves of genus $\leq 2$ in very general sextic 3-folds $X\subseteq\PP^4$, in the case where there is a surjection from $M_H$ onto the normal bundle of the curve in $X$.
\end{remark}

\subsection{Osculating lines}\label{sec-osc}

A $\PP^m$-line (resp. $\PP^n$-line) in $\pmn$ is called an \emph{osculating $\PP^m$-line} (resp. \emph{osculating $\PP^n$-line}) of $X_t$ if it intersects $X_t$ in exactly one point. We denote by $\Lambda^{\PP^m}_a\vert_t$ (resp. $\Lambda^{\PP^n}_b\vert_t$) the locus on $X_t$ swept out by osculating $\PP^m$-lines (resp. osculating $\PP^n$-lines). First, we show the following by using modifications of Lemmas 2.10--11 in \cite{CoskunRiedlsurfaces}:
\begin{enumerate}
    \item When $a=2m+n-2$ and $s_1=m+n-2$, the curve $h_t: Y_t\rightarrow X_t$ lies in $\Lambda^{\PP^m}_a\vert_t$.
    \item When $b=m+2n-2$ and $s_2=m+n-2$, the curve $h_t: Y_t\rightarrow X_t$ lies in $\Lambda^{\PP^n}_b\vert_t$.
\end{enumerate}
Then, we show that $\Lambda^{\PP^m}_a\vert_t$ and $\Lambda^{\PP^n}_b\vert_t$ are curves of general type when $a,b$ are large enough. Note that the hypothesis of the following lemma is equivalent to the condition that the curve $h_t:Y_t\rightarrow X_t$ is of type $(s_1,s_2)=(m+n-2,0)$ and $m+n\geq4.$

\begin{lemma}[cf Lemma 2.10 in \cite{CoskunRiedlsurfaces}]\label{lem-2.10}
    Suppose that at a general point $(p,t) \in \yY$,
    \begin{enumerate}
        \item $M_{H_1}\rightarrow N_{h_t/X_t}$ has rank-one image for a generic polynomial in $H^0(\eE\otimes H_1^\vee)$ and is not surjective, and 
        \item $M_{H_2}\rightarrow N_{h_t/X_t}$ has torsion image for a generic polynomial in $H^0(\eE\otimes H_2^\vee).$
    \end{enumerate}
    Then there exists a $\PP^m$-line $p\in \ell\subseteq \pmn$ whose ideal $H^0(\eE\otimes \II_\ell)$ is contained in $T_{\yY/\pmn}\vert_{(p,t)}$. Moreover, there is only one such $\PP^m$-line $\ell$, and it is the line induced by the generic rank-one image $Q$ of $M_{H_1}$ in $N_{h_t/X_t}$, as in Lemma \ref{lem-scroll}. 
\end{lemma}

\begin{proof}
    Consider the map $M_{H_1}\rightarrow N_{h_t/X_t}$ induced by a generic polynomial in $H^0(\eE\otimes H_1^\vee)$ at the point $p\in Y_t$. Since the map is not surjective, its kernel is independent of the generic polynomial by Lemma 2.2(i) in \cite{Clemens03}. We denote its kernel by $S\subseteq H^0(H_1\otimes \II_p)$ and its image by $Q\subseteq N_{h_t/X_t}\vert_p$.
    By Proposition \ref{prop-main}, $N_{h_t/X_t}\vert_p$ is isomorphic to the quotient $M_\eE\vert_p\;\Big/\;T_{\yY/\pmn}\vert_{(p,t)}.$ Hence by the above assumptions, $T_{\yY/\pmn}\vert_{(p,t)}$ contains the image of the map
    $$\left(S \otimes H^0(\eE\otimes H_1^\vee)\right) \oplus  \left(H^0(H_2\otimes \II_p)\otimes H^0(\eE\otimes H_2^\vee) \right)\rightarrow H^0(\eE\otimes \II_p)\simeq M_\eE\vert_p,$$
    which is the ideal of the $\PP^m$-line containing $p$ that is given by Lemma \ref{lem-scroll}.
    
    Uniqueness follows from a dimension count. We can view $T_{\yY/\pmn}\vert_{(p,t)}\subseteq T_{\xX/\pmn}\vert_{(p,t)}$ as spaces of polynomials. If there are two distinct $\PP^m$-lines whose ideals are contained in $T_{\yY/\pmn}\vert_{(p,t)}$, then we would have $T_{\yY/\pmn}\vert_{(p,t)}=T_{\xX/\pmn}\vert_{(p,t)}$, a contradiction. 
\end{proof}

The following is the analogous result with $H_1$ and $H_2$ switched.

\begin{lemma}\label{lem-2.102}
    Suppose that at a general point $(p,t) \in \yY$,
    \begin{enumerate}
        \item $M_{H_1}\rightarrow N_{h_t/X_t}$ has torsion image for a generic polynomial in $H^0(\eE\otimes H_1^\vee),$ and 
        \item $M_{H_2}\rightarrow N_{h_t/X_t}$ has rank-one image for a generic polynomial in $H^0(\eE\otimes H_2^\vee)$ and is not surjective.
    \end{enumerate}
    Then there exists a $\PP^n$-line $p\in \ell\subseteq \pmn$ whose ideal $H^0(\eE\otimes \II_\ell)$ is contained in $T_{\yY/\pmn}\vert_{(p,t)}$. Moreover, there is only one such $\PP^n$-line $\ell$, and it is the line induced by the generic rank-one image $Q$ of $M_{H_2}$ in $N_{h_t/X_t}$, as in Lemma \ref{lem-scroll2}. 
\end{lemma}

The following result is a generalization of Lemma 2.11 in \cite{CoskunRiedlsurfaces}.

\begin{lemma}[cf Lemma 2.11 in \cite{CoskunRiedlsurfaces}]\label{lem-2.11}
    Let $(p,t)\in \yY$ be a general point and let \linebreak $Z:=p_2^{-1}(p).$ Let $T\subseteq \pmn$ be a subvariety containing $p$ whose ideal $H^0(\eE\otimes \II_T)$ is contained in $T_{\yY/\pmn}\vert_{(p,t)}.$ Then $W:=h(\yY)\cap Z$ is a union of fibers of the restriction map $\beta:Z\rightarrow H^0(\eE\vert_T).$ In particular, $W$ contains the fiber containing $(p,t),$ which is the affine space $(p,t)+H^0(\eE\otimes \II_T).$
\end{lemma} 

\begin{proof}
    The tangent space to a fiber of the above restriction map $\beta$ is $H^0(\eE\otimes \II_T)$ at every point of the fiber, so the fiber containing $(p,t)$ is the affine space $(p,t)+H^0(\eE\otimes \II_T)$. Since $H^0(\eE\otimes \II_T)\subseteq T_{\yY/\pmn}\vert_{(p,t)},$ Lemma 2.11 in \cite{CoskunRiedlsurfaces} implies that $W$ is a union of fibers of the restriction map $\beta:Z\rightarrow H^0(\eE\vert_T).$ Therefore, $W$ contains the fiber containing $(p,t),$ which is $(p,t)+H^0(\eE\otimes \II_T)$.
\end{proof}

Now, we will use Lemmas \ref{lem-2.10} and \ref{lem-2.11} to prove that if $a=2m+n-2$, then a curve of type $(s_1,s_2)=(m+n-2,0)$ lies in the osculation locus $\Lambda^{\PP^m}_a\vert_t\subseteq X_t$.

\begin{lemma}[cf \S2 in \cite{Voisincorrection}, \S3.3 in \cite{Pacienza}]\label{lem-osc}
    Suppose that $a=2m+n-2$, and that at a general point $(p,t) \in \yY$,
    \begin{enumerate}
        \item $M_{H_1}\rightarrow N_{h_t/X_t}$ has rank-one image for a generic polynomial in $H^0(\eE\otimes H_1^\vee)$ and is not surjective, and 
        \item $M_{H_2}\rightarrow N_{h_t/X_t}$ has torsion image for a generic polynomial in $H^0(\eE\otimes H_2^\vee).$
    \end{enumerate}
    Then the curve $h_t:Y_t\rightarrow X_t$ lies in the locus $\Lambda^{\PP^m}_a\vert_t\subseteq X_t$.
\end{lemma}

\begin{proof}
	Let $\ell$ denote the $\PP^m$-line associated to $(p,t)$ obtained from Lemma \ref{lem-2.10}. We can assume without loss of generality that $$p=\{X_1=\ldots=X_m=Y_1=\ldots=Y_n=0\},$$ and  $$\ell=\{X_2=\ldots=X_m=Y_1=\ldots=Y_n=0\}.$$
	Consider the restriction map 
	$$ \beta: T_{\yY/\pmn}\vert_{(p,t)}\rightarrow H^0(\OO_{\ell}(a)(-p)),$$ whose kernel is $H^0(\eE\otimes \II_\ell)$. Denote by $F$ the polynomial corresponding to $t$.\\
	
	Since $\yY$ is $G$-invariant, $T_{\yY/\pmn}\vert_{(p,t)}$ contains the elements of $T_{(\pmn)\times V}\vert_{(p,t)}$ that are tangent to the $G$-orbit of $(p,t)$ and project to 0 in $T_{\pmn}\vert_p$. In particular, $T_{\yY/\pmn}\vert_{(p,t)}$ contains the elements
	$$ F,\; X_1\dfrac{\partial F}{\partial X_0},\;\ldots\;,X_1\dfrac{\partial F}{\partial X_m}.$$
	
	By Lemma \ref{lem-2.11}, we may assume that $F$ is generic in the space $(p,t)+H^0(\eE \otimes \II_\ell),$ so we can assume that the $m-1$ elements
	$$X_1 \frac{\partial F}{\partial X_2}\vert_\ell,\ldots,X_1 \frac{\partial F}{\partial X_m}\vert_\ell$$ in Im $\beta$ are independent modulo the subspace $$K:=\langle\; F\vert_\ell, X_1 \frac{\partial F}{\partial X_0}\vert_\ell, X_1 \frac{\partial F}{\partial X_1}\vert_\ell \;\rangle \subseteq  H^0(\OO_\ell(a)(-p)).$$
	By a dimension count, we have 
	\begin{equation*}
    	\begin{split}
    	    \text{dim Im }\beta &= \dim \left(T_{\yY/\pmn}\vert_{(p,t)}\right)- \dim H^0(\eE\otimes \II_\ell)\\
    	    &=\dim \left(T_{\xX/\pmn}\vert_{(p,t)}\right) - \dim H^0(\eE\otimes\II_\ell)-\dim \left(h^*T_{\xX/\pmn}\vert_{(p,t)}\Big/T_{\yY/\pmn}\vert_{(p,t)}\right)\\
    	    &=a-(m+n-2)\\
    	    &=m
    	\end{split}
	\end{equation*} 
	when $a=2m+n-2.$ Hence, $\dim K \leq \dim \text{Im }\beta-(m-1)= m-(m-1)=1$, which implies that $F\vert_\ell= a\cdot X_1^aY_0^b.$ Therefore, the curve $h_t:Y_t\rightarrow X_t$ lies in the locus $\Lambda^{\PP^m}_a\vert_t\subseteq X_t.$
\end{proof}

The following is the analogous result when $b=m+2n-2$ and the curve is of type $(s_1,s_2)=(0,m+n-2)$ instead.

\begin{lemma}\label{lem-2.112}
    Suppose that $b=m+2n-2$, and that at a general point $(p,t) \in \yY$,
    \begin{enumerate}
        \item $M_{H_1}\rightarrow N_{h_t/X_t}$ has torsion image for a generic polynomial in $H^0(\eE\otimes H_1^\vee),$ and 
        \item $M_{H_2}\rightarrow N_{h_t/X_t}$ has rank-one image for a generic polynomial in $H^0(\eE\otimes H_2^\vee)$ and is not surjective.
    \end{enumerate}
    Then this curve lies in the locus $\Lambda^{\PP^n}_b\vert_t\subseteq X_t$.
\end{lemma}

\begin{remark}\label{rmk-orbit}
    When $n=1$, Lemma \ref{lem-2.112} follows directly from Lemmas 2.10 and 2.11 in \cite{CoskunRiedlsurfaces} by the argument used there toward classifying very general surfaces in $\PP^2\times\PP^1$. (See \cite{CoskunRiedlsurfaces}, Proof of Theorem 3.4, Case 4.) This argument goes as follows. By Lemmas \ref{lem-2.102} and \ref{lem-2.11} above, $W$ is a union of fibers of the restriction map $\beta:Z\rightarrow H^0(\eE\vert_\ell)$ where $\ell$ is the $\PP^n$-line $\pi_1^{-1}\pi_1(p)$, so $X_t\cap \ell$ is a subscheme $B$ of $b$ (not necessarily distinct) points on $\PP^1$. If $B$ contained $\geq 2$ distinct points, then its $GL(2)$-orbit would have codimension $\leq b-2$, which equals $m-2$ when $b=m+2n-2=m$. Since $\yY$ is $G$-invariant, $W$ contains the fibers of $\beta$ over $B$, so we would have 
    $\codim(\yY\subseteq \xX)\leq m-2.$ However, this is a contradiction since $\codim(\yY\subseteq \xX)= m+n-2=m-1$ by (\ref{eqn-codimyx}). Hence, $B$ has to be one unique point with multiplicity $b$. 
\end{remark}

For the rest of this section, we investigate the geometry of the osculation loci $\Lambda^{\PP^m}_a\vert_t$ and $\Lambda^{\PP^n}_b\vert_t$. Let us define the incidence variety
$$\Delta^{\PP^m}_a =\{ (L,p,t)\mid \pi_1^{-1}(L)\cap \pi_2^{-1}(\pi_2(p)) \cap X_t = a\cdot [p] \} \subseteq \GG(1,m)\times(\pmn)\times V,$$ and denote its fiber over $t$ by $\Delta^{\PP^m}_a\vert_t.$

\begin{lemma}[cf Theorem 5.2 in \cite{Clusfam}]\label{lem-gentype}
	For a very general $t\in U$, $\Delta^{\PP^m}_a\vert_t$ is a smooth, irreducible variety of dimension $2m-1+n-a$. Moreover, it is of general type when $a\geq \max\{\sqrt{2m}+1,3\}$ and $b\geq n+2$.
\end{lemma}

\begin{proof}
    Since $\Delta^{\PP^m}_a$ is a vector bundle over the space of pairs $(L,p)\in \GG(1,m)\times(\pmn)$ such that $\pi_1(p)\in L$, it is a smooth and irreducible variety. 
	
	Now we prove that the canonical bundle $K_{\Delta^{\PP^m}_a}$ is very ample when $a,b$ are large enough. For convenience, let us denote $\GG:=\GG(1,m)\times (\PP^m\times \PP^n)\times V$, and let us denote the various projection maps by $\rho_1: \GG\rightarrow \GG(1,m)$, $\rho_2: \GG\rightarrow \PP^m\times \PP^n$, $\rho_{2,m}: \GG\rightarrow \PP^m$, $\rho_{2,n}: \GG\rightarrow \PP^n$, and $\rho_{3}: \GG\rightarrow V$. We will construct $\Delta^{\PP^m}_a$ as the vanishing locus of some vector bundle on $\GG$. Let $\sigma,H_1,H_2$ be the pullback of the hyperplane divisors via $\rho_1,\rho_{2,m},\rho_{2,n}$ respectively. Then the canonical bundle of $\GG$ is of class
	$$K_\GG=(-m-1)\sigma+(-m-1)H_1+(-n-1)H_2=:(-m-1,-m-1,-n-1).$$
	First consider the natural map $$e: \rho_3^*\mathcal{O}_V(-1)\rightarrow \rho_2^*\mathcal{E}$$ of vector bundles on $\GG$. Its zero scheme is the locus of points $(L,p,t)$ with $p\in X_t.$ Then consider the natural map 
	$$f:\rho_{2,m}^*\mathcal{O}_{\PP^m}(-1)\hookrightarrow\rho_1^*\mathcal{O}^{m+1}_{\mathbb{G}(1,m)}\twoheadrightarrow \rho_1^*Q,$$ where the surjection comes from the following tautological sequence on $\GG(1,m)$
	$$0\rightarrow S \rightarrow \mathcal{O}^{m+1}_{\mathbb{G}(1,m)}\rightarrow Q \rightarrow 0. $$
	The zero scheme of $f$ is the locus of points $(L,p,t)$ with $\pi_1(p)\in L.$ Since $\Delta^{\PP^m}_1$ is the common zero scheme of $e,f$, we can compute the class of its canonical bundle as follows:
	\begin{equation}\label{eqn-delta1}
	    \begin{split}
	        K_{\Delta^{\PP^m}_1} &= K_\GG + c_1(\rho_2^*\eE) + c_1(\rho_{1}^*Q\otimes \rho_{2,m}^*\OO_{\PP^m}(1))\\
	        &=K_\GG + c_1(\rho_2^*\eE) + c_1(\rho_{1}^*Q)+ (m-1)\cdot c_1(\rho_{2,m}^*\OO_{\PP^m}(1)))\\
	        &=(-m-1, -m-1, -n-1) + (0, a, b) + (1, 0, 0) + (0, m-1, 0)\\
	        &=(-m, a-2, b-n-1).
	    \end{split}
	\end{equation}
	
	Note that on $\Delta^{\PP^m}_1$, the natural map 
	$$\rho_{2,m}^*\OO_{\PP^m}(-1) \hookrightarrow \rho_1^*\OO^{m+1}_{\GG(1,m)}$$ from above factors as 
	\begin{center}
		\begin{tikzcd}
		&\rho_{2,m}^*\OO_{\PP^m}(-1) \arrow[rr, hookrightarrow] \arrow[dr, hookrightarrow, "\mu"]& &\rho_1^*\OO^{m+1}_{\GG(1,m)}\\
		& &\rho_1^*S  \arrow[ur, hookrightarrow] &
		\end{tikzcd}.
	\end{center}
	Let $R^\vee$ be the cokernel of $\mu$. We have the short exact sequence 
	\begin{equation*}
	    0 \rightarrow \rho_{2,m}^*\OO_{\PP^m}(-1) \xrightarrow{\mu} \rho_1^*S \rightarrow R^\vee \rightarrow 0,
	\end{equation*}
	on $\Delta^{\PP^m}_1$, which induces a filtration $F^\bullet$ on $\rho_1^*Sym^aS^\vee$ with
	$$ F^i/F^{i+1}= \rho_{2,m}^*\OO_{\PP^m}(a-i)\otimes R^i.$$ 
	Note that 
	\begin{equation}\label{eqn-filtration}
	    c_1(F^i/F^{i+1}) = (0, a-i, 0) + i\cdot (1, -1, 0) = (i,a-2i,0).
	\end{equation}
	
	Since $\Delta^{\PP^m}_a$ is the zero scheme of the (well-defined) natural map 
	$$\OO_V(-1) \rightarrow  F^1/F^a,$$
	we can compute its canonical bundle using adjunction again as follows:
	\begin{align*}
	K_{\Delta^{\PP^m}_a} &= K_{\Delta^{\PP^m}_1}+c_1(F^1/F^a)\\
	&= K_{\Delta^{\PP^m}_1}+\sum_{i=1}^{a-1}c_1(F^i/F^{i+1})\\
	&=(-m,a-2,b-n-1)+\sum_{i=1}^{a-1}(i,a-2i,0) \qquad \text{by (\ref{eqn-delta1}) and (\ref{eqn-filtration})}\\
	&=\left(-m+\frac{a(a-1)}{2},a-2,b-n-1\right).
	\end{align*}
	Hence $K_{\Delta^{\PP^m}_a}$ is very ample when $a\geq \max\{\sqrt{2m}+1,3\}$ and $b\geq n+2$.
	
	Finally we verify the dimension of $\Delta^{\PP^m}_a\vert_t.$ Since a very general fiber of the projection map $(L,p,t)\mapsto (L,p)$ has codimension $a$ in $V$, we have
	$$\dim \Delta^{\PP^m}_a= \dim \GG(1,m) + n + 1 + (\dim V - a) = 2m-1+n-a+\dim V,$$ 
	and so $\dim \Delta^{\PP^m}_a\vert_t=2m-1+n-a$.
\end{proof}

\begin{lemma}[cf Proposition 5.3 in \cite{Clusfam}]\label{lem-curve}
    For a very general $t\in U,$ the projection map $\Delta^{\PP^m}_a\vert_t\rightarrow \Lambda^{\PP^m}_a\vert_t$ is an isomorphism when $a\geq 2m+n-2.$
\end{lemma}

\begin{proof}
    Let us define 
    $$\DD=\{(p,t)\mid p\in X_t \text{ and } \exists L_1\neq L_2\in \GG(1,m) \text{ s.t. }L_1\cap L_2=\pi_1(p) \} \subseteq \xX_V.$$ 
    From the proof of Lemma \ref{lem-negative}, $X_t$ does not contain $\PP^m$-lines when $a\geq 2m+n-2$, so it is enough to prove that codim $(\DD\subseteq \xX_V)\geq m+n.$
    
    Define the incidence variety 
    $$ I=\{(p, t,L_1,L_2)\mid p\in X_t, L_1\neq L_2 \text{ and } L_1\cap L_2=\pi_1(p) \}. $$
    The space of tuples $(p,L_1,L_2)$ such that $L_1\cap L_2=\pi_1(p)$ has dimension equal to $$2(m-1)+(1+n)+(m-1)=3m+n-2.$$ Since a fiber of $I$ over such a tuple has dimension $\dim V-2a$, we have $$\dim I = \dim V+3m+n-2-2a,$$ so $\DD$ being the image of $I$ must have dimension $\leq \dim V+3m+n-2-2a$. Therefore,
    \begin{equation*}
        \begin{split}
            \text{codim }(\DD\subseteq \xX_V) &\geq \dim \xX_V - (N+3m+n-2-2a )\\
            &=(\dim V+m+n-1) - (\dim V+3m+n-2-2a )\\
            &=2a-2m+1,
        \end{split}
    \end{equation*}
    which is $\geq m+n$ when $a\geq \frac{3m+n-1}{2}.$
\end{proof}

The following are the analogous results to Lemmas \ref{lem-gentype} and \ref{lem-curve} for $\Delta^{\PP^n}_b\vert_t$ and $\Lambda^{\PP^n}_b\vert_t.$

\begin{lemma}
	For a very general $t\in U$, $\Delta^{\PP^n}_b\vert_t$ is a smooth, irreducible variety of dimension $m-1+2n-b$. Moreover, it is of general type when $a\geq m+2$ and $b\geq \max\{\sqrt{2n}+1,3\}$.
\end{lemma}

\begin{lemma}
    For a very general $t\in U,$ the projection map $\Delta^{\PP^n}_b\vert_t\rightarrow \Lambda^{\PP^n}_b\vert_t$ is an isomorphism when $b\geq m+2n-2.$
\end{lemma}

One can use the analogues of the above tools in the $\PP^n$ setting to show that a very general hypersurface $X_d\subseteq \PP^n$ of degree $d\geq 2n-2$ is algebraically hyperbolic when $n\geq5$. Our proof is built upon the proofs of \cite{Pacienza2} and \cite{ClemensRan} for the $n\geq 6$ case. The key to our improved result is the use of Coskun and Riedl's scroll argument from \cite{CoskunRiedl}, which also gives an alternate proof of Voisin's result on the algebraic hyperbolicity of very general hypersurfaces of degree $d\geq 2n-1$ when $n\geq 4$ \cite{Voisin, Voisincorrection}.

\begin{theorem}\label{thm-pn}
   A very general hypersurface $X\subseteq \PP^n$ of degree $d$ is algebraically hyperbolic when \emph{(1)} $n\geq 4$ and $d\geq 2n-1$; or \emph{(2)} $n\geq 5$ and $d\geq 2n-2$.
\end{theorem}

\begin{proof}
    Suppose $X$ is parametrized by $t\in U$ and let $h_t:Y_t\rightarrow X_t$ be a curve in $X.$ Then, there is a surjection 
    $$ \alpha: M_H\;^{\oplus s}\rightarrow N_{h_t/X_t}$$ 
    for some integer $s$. Take $s$ to be the minimum. Since rank $N_{h_t/X_t}=n-2$, we have $s\leq n-2.$ 
    
    Suppose that $n\geq 4$. If $s\leq n-3,$ then 
    $$\deg N_{h_t/X_t}\geq (n-3)\cdot \deg \alpha\left(M_{H}\right)\geq (n-3)\cdot -\deg H\vert_{Y_t}.$$ Therefore, 
	\begin{equation*}
	    \begin{split}
	        2g(Y_t)-2 &= (K_{X_t}\cdot Y_t) + \deg N_{h_t/X_t}\\
	        &\geq ((d-n-1)-(n-3))\cdot (Y_t\cdot H) \\
	        &=(d-(2n-2))\cdot (Y_t\cdot H) 
	    \end{split}
	\end{equation*}
	which is at least $(Y_t\cdot H)$ when $d\geq 2n-1$. If $s=n-2,$ then the map $$M_H\rightarrow N_{h_t/X_t}/\alpha(M_H\;^{\oplus n-3}) $$ given by multiplication by a generic polynomial in $H^0(\PP^n,\OO_{\PP^n}(d-1))$ has a rank-one image $Q$, which induces a surface scroll $\Sigma\subseteq \PP^n$ of degree $\deg Q + (Y_t\cdot H)$. Since the curve $Y_t$ is contained in the curve $\Sigma\cap X_t$ whose degree is $d\cdot (\deg Q + (Y_t\cdot H)),$ we must have $$\deg Q\geq \left(\frac{1}{d}-1 \right)\cdot (Y_t\cdot H).$$
    Therefore, 
	\begin{equation*}
	    \begin{split}
	        2g(Y_t)-2 &= (K_{X_t}\cdot Y_t) + \deg N_{h_t/X_t}\\
	        &\geq ((d-n-1)-(n-3))\cdot (Y_t\cdot H) +\deg Q\\
	        &\geq(d+\frac{1}{d}-(2n-1))\cdot (Y_t\cdot H) 
	    \end{split}
	\end{equation*}
	which is at least $\frac{1}{d}(Y_t\cdot H)$ when $d\geq 2n-1$. This concludes the proof for $n\geq 4$ and $d\geq 2n-1$.
	
	Now suppose that $n\geq5$ and $d= 2n-2.$ If $s\leq n-4$, then as in the case $s\leq n-3$ from above, we have
    $\deg N_{h_t/X_t}\geq (n-4)\cdot -\deg H\vert_{Y_t}$. Therefore, $$2g(Y_t)-2 \geq ((d-n-1)-(n-4))\cdot (Y_t\cdot H) =(d-(2n-3))\cdot (Y_t\cdot H)=(Y_t\cdot H).$$ If $s=n-3$, then as in the case $s= n-2$ from above, the map 
    $$ M_H\rightarrow N_{h_t/X_t}/\alpha(M_H\;^{\oplus n-4}) $$ has a rank-one image $Q$, since we assume that $n\geq 5$. Then it follows from the scroll method as applied above that $2g(Y_t)-2\geq \frac{1}{d}(Y_t\cdot H)$. If $s=n-2$, then the map 
    $$ M_H\rightarrow N_{h_t/X_t}$$ has rank-one image and is not surjective. By the analogues of Lemmas \ref{lem-2.10}, \ref{lem-2.11} and \ref{lem-osc} for $\PP^n$, the curve lies in the 1-osculation locus of $X_t$, which is a curve of general type by \S4.1 in \cite{Pacienza}.
\end{proof}

\begin{remark}\label{rem-sextic}
    Consider a curve in a very general sextic threefold, i.e. when $n=4$ and $d=2n-2$. If the curve is of type $s=2$, then it satisfies the inequality (\ref{eqn-definition}) by the same argument as the case $n\geq 5$, $d=2n-2$ and $s=n-2$ in Theorem \ref{thm-pn}. If the curve is of type $s=1$, i.e. the map 
     $M_H\rightarrow N_{h_t/X_t}$ induced by a generic degree 5 polynomial has a rank-two image, then the scroll method as above does not apply (see also Remark \ref{rmk-lines}).
\end{remark}

\section{Proof of Theorem \ref{thm-pmxpn}}\label{sec-main}

\subsection{Proof of Theorem \ref{thm-pmxpn}(2).}

We verify that very general hypersurfaces of low bidegrees $(a,b)$ contain lines or elliptic curves. Hence, they are not algebraically hyperbolic. 

\begin{lemma}\label{lem-negative}
	Suppose that $a\leq 2m+n-3$ or $b\leq m+2n-3$. Then a very general hypersurface $X\subseteq \PP^m\times \PP^n$ of bidegree $(a,b)$ contains a line. In particular, it is not algebraically hyperbolic.
\end{lemma}

\begin{proof}
	View $X$ as a very general $\PP^n$-family of degree $a$ hypersurfaces in $\PP^m$ via the projection map $\pi_2:X\rightarrow \PP^n$. The space of degree $a$ hypersurfaces in $\PP^m$ containing a line has codimension $(a+1)-2(m-1)$ in $H^0(\PP^m,\OO_{\PP^m}(a))$, so it has non-trivial intersection with $X$ when $a\leq 2m+n-3$. Therefore $X$ contains a $\PP^m$-line when $a\leq 2m+n-3$. The same argument shows that $X$ contains a $\PP^n$-line when $b\leq m+2n-3$.
\end{proof}

\begin{lemma}\label{lem-p2xp2}
	Let $(m,n)=(2,2)$. Suppose that $a=4$ or $b=4$. Then a very general hypersurface $X\subseteq \PP^2\times \PP^2$ of bidegree $(a,b)$ contains an elliptic curve. In particular, it is not algebraically hyperbolic.
\end{lemma}

\begin{proof}
	As above, we can view $X$ as a very general net of degree 4 curves in $\PP^2$. By the degree-genus formula for plane curves, a general member of this net has geometric genus 3. We prove that there exists a curve in this net that has two nodes. Such a curve would be an elliptic curve since its geometric genus is $3-2=1$. Requiring a quartic plane curve to contain two particular distinct nodes in $\PP^2$ imposes $2\times 3 =6$ conditions on the equation of the curve. Since the two points can vary over a $(2\times 2 =4)$-dimensional family, the space of quartic plane curves containing any two distinct nodes has codimension $6-4=2$ in $H^0(\PP^2,\mathcal{O}_{\PP^2}(4))$. Hence, this family intersects the two-dimensional family that sweeps out the very general hypersurface $X$.
\end{proof}

\subsection{Proof of Theorem \ref{thm-pmxpn}(1)}

We divide up the proof by $(m,n)$. Then, for each possible type $(s_1,s_2),$ we produce a uniform $\varepsilon > 0$ such that all curves $h_t:Y_t\rightarrow X_t$ of type $(s_1,s_2)$ (see Definition \ref{def-s1s2}) lying in a very general hypersurface $X_t\subseteq \pmn$ satisfies
\begin{equation}\label{eqn-definition2}
    2g(Y_t)-2\geq \varepsilon \cdot ((Y_t\cdot H_1)+(Y_t\cdot H_2))
\end{equation}
provided that the bidegree $(a,b)$ of $X_t$ lies in the stated range. We assume without loss of generality that $s_1\geq s_2.$

\subsubsection{Proof of Theorem \ref{thm-pmxpn}(1)(a).}\label{sec-a}Let $m+n\geq5$.

\textbf{Case 1:} Suppose that $s_1,s_2\leq m+n-4$. Then, by (\ref{eqn-bounds2}), we have
\begin{equation*}
    \begin{split}
	        &2g(Y_t)-2\\ 
	        &\geq ((a-m-1)-(m+n-4))\cdot (Y_t\cdot H_1) + ((b-n-1)-(m+n-4))\cdot (Y_t\cdot H_2)\\
	        &=(a-(2m+n-3))\cdot (Y_t\cdot H_1) + (b-(m+2n-3))\cdot (Y_t\cdot H_2),
        \end{split}
\end{equation*}
which is at least $(Y_t\cdot H_1)+(Y_t\cdot H_2)$ when $a\geq 2m+n-2$ and $b\geq m+2n-2.$ 

\textbf{Case 2:} Suppose that $s_1= m+n-3$, so $s_2 \leq 1$. Consider the map
$$M_{H_1}\rightarrow N_{h_t/X_t}\Big/\alpha(M_{H_1}^{\oplus m+n-4})$$ 
induced by $\alpha$. Since $m+n\geq5$, this map must have a rank-one image $Q$, so we can apply the scroll considerations of Lemma \ref{lem-scroll}. Let $\Sigma\subseteq \pmn$ denote the $\PP^m$-scroll given by $Q$. The $H_1$-degree of $\Sigma$ is $\Sigma\cdot H_1^2$, which is equal to $\deg Q + (Y_t\cdot H_1)$ by Lemma \ref{lem-scroll}. Since $\Sigma$ is a $\PP^m$-scroll, we have $\Sigma\cdot H_2=(Y_t\cdot H_2)\cdot H_1^{m-1}H_2^n.$ Hence, $\Sigma$ is of numerical class
$$ (Y_t\cdot H_2) \cdot H_1^{m-1}H_2^{n-1}+ \left(\deg Q + (Y_t\cdot H_1)\right)\cdot H_1^{m-2}H_2^n.$$
$Y_t$ is contained in the curve $\Sigma \cap X_t$ whose class is $$(a\cdot  (H_2\cdot Y_t))\cdot H_1^mH_2^{n-1}+(b\cdot (H_2\cdot Y_t)+a\cdot (\deg Q + (H_1\cdot Y_t)))\cdot H_1^{m-1}H_2^n.$$ Since $\Sigma \cap X_t- Y_t$ is effective, we have $H_1\cdot (\Sigma \cap X_t-Y_t)\geq0$. Hence, we obtain the inequality $$b\cdot (H_2\cdot Y_t)+a\cdot (\deg Q + (H_1\cdot Y_t))-(H_1\cdot Y_t)\geq0,$$ which can be rearranged into $$\deg Q\geq \left(\frac{1}{a}-1\right)\cdot (H_1\cdot Y_t)+\left(-\frac{b}{a}\right)\cdot (H_2\cdot Y_t).$$
Therefore, by Lemma \ref{lem-bounds},
\begin{equation*}
\begin{split}
    &2g(Y_t)-2\\ 
    &\geq ((a-m-1)-(m+n-4))\cdot (Y_t\cdot H_1) + ((b-n-1)-1)\cdot (Y_t\cdot H_2)+ \deg Q\\
    &\geq \left(a+\frac{1}{a}-(2m+n-2)\right)\cdot (H_1\cdot Y_t)+\left(b-\frac{b}{a}-n-2\right)\cdot (H_2\cdot Y_t),
\end{split}
\end{equation*}
which is at least some multiple of $(Y_t\cdot H_1)+(Y_t\cdot H_2)$ when $a\geq 2m+n-2$ and $b>(n+2)\cdot \left(\dfrac{a}{a-1}\right)$. One can check that $b\geq m+2n-2>(n+2)\cdot \left(\dfrac{a}{a-1}\right)$ when $m+n\geq5$ and $m\geq n.$

\textbf{Case 3:} Suppose that $s_1=m+n-2$, so $s_2 =0.$ Then, the map 
$$M_{H_1}\rightarrow N_{h_t/X_t}$$ induced by $\alpha$ has a rank-one image $Q$, so we can apply the scroll considerations as in Case 2 to obtain $$\deg Q\geq \left(\frac{1}{a}-1\right)\cdot (H_1\cdot Y_t)+\left(-\frac{b}{a}\right)\cdot (H_2\cdot Y_t).$$ Therefore, by Lemma 3.3,
\begin{equation*}
\begin{split}
    &2g(Y_t)-2\\ 
    &\geq ((a-m-1)-(m+n-3))\cdot (Y_t\cdot H_1) + (b-n-1)\cdot (Y_t\cdot H_2)+ \deg Q\\
    &\geq \left(a+\frac{1}{a}-(2m+n-1)\right)\cdot (H_1\cdot Y_t)+\left(b-\frac{b}{a}-n-1\right)\cdot (H_2\cdot Y_t),
\end{split}
\end{equation*}
which is at least some multiple of $(Y_t\cdot H_1)+(Y_t\cdot H_2)$ when $a\geq 2m+n-1$ and $b>(n+1)\cdot \left(\dfrac{a}{a-1}\right)$. One can check that $b\geq m+2n-2>(n+1)\cdot \left(\dfrac{a}{a-1}\right)$ when $m+n\geq5$ and $m\geq n.$ Hence, we focus on the case when $a=2m+n-2$ and $b\geq m+2n-2.$ Suppose that the hypothesis of Lemma \ref{lem-2.10} is satisfied. Then, by Lemma \ref{lem-osc}, the curve $h_t:Y_t\rightarrow X_t$ lies in the subvariety $\Lambda^{\PP^m}_a\vert_t\subseteq X_t$. Lemmas \ref{lem-gentype} and \ref{lem-curve} imply that $\Lambda^{\PP^m}_a\vert_t$ is a smooth, irreducible curve of general type. Therefore, there is some $\varepsilon>0$ such that (\ref{eqn-definition2}) is satisfied. \qed

\subsubsection{Proof of Theorem \ref{thm-pmxpn}(1)(b).}Let $(m,n)=(2,2)$.

\textbf{Case 1:} Suppose that $s_1,s_2\leq 1.$ We proceed as in Case 1 of \S\ref{sec-a} and obtain the inequality
\begin{equation*}
    2g(Y_t)-2\geq (a-4)\cdot (Y_t\cdot H_1) + (b-4)\cdot (Y_t\cdot H_2),
\end{equation*}
which is at least $(Y_t\cdot H_1)+(Y_t\cdot H_2)$ when $a,b\geq 5.$ 

\textbf{Case 2:} Suppose that $s_1= 2$, so $s_2 =0$. We proceed as in Case 2 of \S\ref{sec-a}. Let $Q$ be the rank-one image of $\alpha:M_{H_1}\rightarrow N_{h_t/X_t},$ and let $\Sigma\subseteq \pmn$ denote the $\PP^m$-scroll given by $Q$. By Lemma \ref{lem-scroll} and the same calculations as above, we have $$\deg Q\geq \left(\frac{1}{a}-1\right)\cdot (H_1\cdot Y_t)+\left(-\frac{b}{a}\right)\cdot (H_2\cdot Y_t),$$
and so 
\begin{equation*}
\begin{split}
    2g(Y_t)-2
    &\geq ((a-3)-1)\cdot (Y_t\cdot H_1) + (b-3)\cdot (Y_t\cdot H_2)+ \deg Q\\
    &\geq \left(a+\frac{1}{a}-5\right)\cdot (H_1\cdot Y_t)+\left(b-\frac{b}{a}-3\right)\cdot (H_2\cdot Y_t),
\end{split}
\end{equation*}
which is at least some multiple of $(Y_t\cdot H_1)+(Y_t\cdot H_2)$ when $a,b\geq5.$ \qed

\subsubsection{Proof of Theorem \ref{thm-pmxpn}(1)(c).}\label{sec-c}Let $(m,n)=(3,1)$.

\textbf{Case 1:} Suppose that $(s_1,s_2)=(1,0).$ We proceed as in Case 1 of \S\ref{sec-a} and obtain the inequality
\begin{equation*}
    2g(Y_t)-2\geq (a-5)\cdot (Y_t\cdot H_1) + (b-2)\cdot (Y_t\cdot H_2),
\end{equation*}
which is at least $(Y_t\cdot H_1)+(Y_t\cdot H_2)$ when $a\geq6$ and $b\geq3.$ 

\textbf{Case 2:} Suppose that $(s_1,s_2)=(1,1).$ We proceed as in Cases 2 and 3 of \S\ref{sec-a}. For $i=1,2$, let $Q_i$ be the rank-one image of $M_{H_i}\rightarrow N_{h_t/X_t},$ and let $\Sigma_i\subseteq \pmn$ denote the scroll given by $Q_i$. By Lemma \ref{lem-scroll} and the same calculations as above, we have 
\begin{equation}\label{eqn-q1}
    \deg Q_1\geq \left(\frac{1}{a}-1\right)\cdot (H_1\cdot Y_t)+\left(-\frac{b}{a}\right)\cdot (H_2\cdot Y_t),    
\end{equation}
\begin{equation}\label{eqn-q2}
    \deg Q_2\geq \left(-\frac{a}{b}\right)\cdot (H_1\cdot Y_t)+\left(\frac{1}{b}-1\right)\cdot (H_2\cdot Y_t).    
\end{equation}
Using (\ref{eqn-q1}), we have
\begin{equation*}
\begin{split}
    2g(Y_t)-2
    &\geq (a-4)\cdot (Y_t\cdot H_1) + ((b-2)-1)\cdot (Y_t\cdot H_2)+ \deg Q_1\\
    &\geq \left(a+\frac{1}{a}-5\right)\cdot (H_1\cdot Y_t)+\left(b-\frac{b}{a}-3\right)\cdot (H_2\cdot Y_t),
\end{split}
\end{equation*}
which is at least some multiple of $(Y_t\cdot H_1)+(Y_t\cdot H_2)$ when $a\geq5$ and $b\geq4.$

Using (\ref{eqn-q2}), we have
\begin{equation*}
\begin{split}
    2g(Y_t)-2
    &\geq ((a-4)-1)\cdot (Y_t\cdot H_1) + (b-2)\cdot (Y_t\cdot H_2)+ \deg Q_2\\
    &\geq \left(a-\frac{a}{b}-5\right)\cdot (H_1\cdot Y_t)+\left(b+\frac{1}{b}-3\right)\cdot (H_2\cdot Y_t),
\end{split}
\end{equation*}
which is at least some multiple of $(Y_t\cdot H_1)+(Y_t\cdot H_2)$ when $a\geq8$ and $b\geq3.$

\textbf{Case 3:} Suppose that $(s_1,s_2)=(2,0).$ Then the same proof as Case 4 of \S\ref{sec-a} applies to show that there is some $\varepsilon>0$ such that (\ref{eqn-definition2}) is satisfied when $a\geq5$ and $b\geq3.$\qed

\begin{remark}\label{rmk-open}
    \S\ref{sec-c} shows where our proof techniques could not resolve the remaining open cases in $\PP^3\times \PP^1$. Specifically, our techniques could not produce an $\varepsilon>0$ such that curves of some types $(s_1,s_2)$ satisfy (\ref{eqn-definition2}), hence we could not confirm algebraic hyperbolicity. We list such types of curves for each open case:
    \begin{enumerate}
        \item For $(a,b)=(5,3)$, curves of types $(s_1,s_2)=(1,0)$ or $(1,1)$;
        \item For $(a,b)=(5,b)$ with $b\geq4$, curves of type $(s_1,s_2)=(1,0)$;
        \item For $(a,b)=(6,3)$ or $(7,3)$, curves of type $(s_1,s_2)=(1,1)$.
    \end{enumerate}
\end{remark}

\bibliographystyle{plain}

\end{document}